\documentclass[11pt]{amsart}

\parskip=\smallskipamount

\newtheorem{theorem}{Theorem}[section]
\newtheorem{lemma}[theorem]{Lemma}
\newtheorem{corollary}[theorem]{Corollary}
\newtheorem{proposition}[theorem]{Proposition}

\theoremstyle{definition}
\newtheorem{definition}[theorem]{Definition}

\newtheorem{remark}[theorem]{Remark}

\newcommand{\B}{\mathbb{B}}
\newcommand{\C}{\mathbb{C}}

\newcommand{\R}{\mathbb{R}}

\newcommand{\Aut}{\mathop{{\rm Aut}}}

\numberwithin{equation}{section}

\begin{document}
\title{Carleman approximation by holomorphic automorphisms of $\mathbb C^n$}
\author{Frank Kutzschebauch}
\address{Frank Kutzschebauch: Mathematisches Institut, Universit\"{a}t Bern, Sidlerstr. 5, CH-3012, Bern, Switzerland}
\email{Frank.Kutzschebauch@math.unibe.ch}
\author{Erlend Forn\ae ss Wold}
\address{Erlend Forn\ae ss Wold: Matematisk Institutt, Universitetet i Oslo, Postboks 1053 Blindern, 0316 Oslo}
\email{erlendfw@math.uio.no}

%
%    General info
%
\date{\today}

\thanks{The research of the first author was partially supported by Schweizerischer Nationalfonds grant No 20021-140235/1}
\thanks{The research of the second author was partially supported by grant NFR-209751/F20 from the Norwegian Research Council}
\begin{abstract}
We approximate smooth maps defined on non-compact totally real manifolds by 
holomorphic automorphisms of $\mathbb C^n$. 
\end{abstract}

\maketitle

\section{Introduction}

The aim of the present paper is to prove a version of the Anders\'en-Lempert Theorem with control on \emph{non-compact} totally  real submanifolds of $\C^n$.  We use coordinates $z_j=x_{2j-1}+ix_{2j}$ on $\mathbb C^n$, and by $\mathbb R^s\subset\mathbb C^n$
we mean $\{z\in\mathbb C^n;x_{2j-1}=0 \mbox{ for } j>s, x_{2j}=0 \mbox{ for } j\geq 1\}$.  
The following is our main result (see also Theorem \ref{FRCarleman} for a more general statement).

\begin{theorem}\label{main}
Let $K\subset\mathbb C^n$ be a compact set, let $\Omega$ be an open set containing $K$, and let 
$$
\phi:[0,1]\times (\Omega\cup\mathbb R^s)\rightarrow\mathbb C^n, s<n,
$$
be an isotopy of smooth embeddings, $\phi_0=\phi(0,\cdot)=id$, such that the following hold:
\begin{itemize}
\item[1)] $\phi_t|_\Omega$ is holomorphic for all $t$,
\item[2)] $\phi_t(K\cup\mathbb R^s)$ is polynomially convex for all $t$, and 
\item[3)] there is some compact set $C\subset\mathbb R^s$ such that 
$\phi_t|_{\mathbb R^s\setminus C}=id$ for all $t$.
\end{itemize}
Then for any $k\in\mathbb N$  we have that $\phi_1$ is $\mathcal C^k$-approximable on  $K\cup\mathbb R^s$,
in the sense of Carleman, by holomorphic automorphisms of $\mathbb C^n$.
\end{theorem}

In Proposition \ref{fixingisotopy} we give a result to the effect that if we assume that $\phi_t(K)$
is polynomially convex for all $t$, then there are arbitrarily small perturbations 
of $\phi_t$ achieving 2).

As a first application of this result we also prove the following:

\begin{theorem}
 Let $\phi:\mathbb R^s\rightarrow \mathbb R^s$ be a smooth 
automorphism, and assume that $s<n$.
Then $\phi$ can be approximated in the fine Whitney topology by holomorphic automorphisms of $\mathbb C^n$.
\end{theorem}

Our main Theorem \ref{main} generalizes work of Forstneri\v{c}-Rosay \cite{FR}, Forstneri\v{c} \cite{Forstneric1994}
and Forstneri\v{c}-L\o w-\O vrelid \cite{FLO}, where similar results were proved 
for compact totally real manifolds.  The proof of the main theorem depends on the Anders\'{e}n-Lempert
theory and also results on Carleman approximation by entire \emph{functions}.  Carleman 
approximation by entire functions on $\mathbb R^s\subset\mathbb C^n$ was 
proved by Scheinberg \cite{Scheinberg} and Hoischen \cite{Hoischen}.  On smooth 
curves in $\mathbb C^n$ it was proved by Alexander \cite{Alexander}, and 
more generally, for dendrites, it was proved by Gauthier and Zeron \cite{GauthierZeron}.

\section{Prelimiaries}

\begin{definition}
For a function $f\in\mathcal C^k(\mathbb C^n)$ we let $|f|_{k,z}$ denote the pointwise semi-norm
$$
|f|_{k,z}:=\underset{|\alpha|\leq k}{\sum}|\frac{d^{\alpha}f}{dx^{\alpha}}(z)|.
$$
If $\phi=(f_1,...,f_m):\mathbb C^n\rightarrow\mathbb C^m$ is a $\mathcal C^k$-smooth map, we define 
$$
|\phi|_{k,z}:=\underset{1\leq j \leq m}{\sum}|f_j|_{k,z}
$$
\end{definition}

In this article we will be interested to approximate $\mathcal C^\infty$-smooth maps 
$$
\phi:M\subset\mathbb C^n\rightarrow\mathbb C^n
$$
by holomorphic automorphisms of $\mathbb C^n$, with respect to the norm $|\cdot|_{k,x}$ along $M$, 
where $M$ is a non-compact totally real embedded submanifold of $\mathbb C^n$.  Since holomorphic
functions satisfy the Cauchy-Riemann equations, this clearly requires some assumptions on $\phi$.  We will work 
in the following class of maps.    

\begin{definition}
Let $A$ be a subset of $\mathbb C^n$ and let $\phi:\mathbb C^n\rightarrow\mathbb C^m$
be a smooth mapping of class $\mathcal C^k$.  We will write 
$\phi\in\mathcal H_k(\mathbb C^n,A)$ if any component function $f$ of $\phi$ has
the property that $\overline\partial f$ vanishes to order $k-1$ on $A$,
\emph{i.e.}, $\overline\partial\frac{d^\alpha f}{dx^\alpha}(z)=0$ 
for all $\|\alpha\|<k$ and all $z\in A$, $z=(z_1,...,z_n), z_j=x_{2j-1}+i\cdot x_{2j}$.
\end{definition}

\begin{definition}
Let $K\subset\mathbb C^n$ be compact.  As usual we let 
$$
\widehat K:=\{z\in\mathbb C^n; |f(z)|\leq\|f\|_K, \forall f\in\mathcal O(\mathbb C^n)\}
$$
denote the \emph{polynomially convex hull of $K$}.  If $K\subset\mathbb C^n$ 
is any closed subset we write $K$ as an increasing union $K=\underset{j\in\mathbb N}{\cup}K_j$
of compact sets $K_j$ with $K_j\subset K_j^{\circ}$, and define the \emph{holomorphically convex hull} of $K$ by
$$
\widehat K:=\underset{j\in\mathbb N}{\cup}\widehat K_j.
$$
We define a set function $h$ defined on closed subsets $K$ of $\mathbb C^n$
by 
$$
h(K):=\widehat K\setminus K.
$$
We say that a closed set $K\subset\mathbb C^n$ has \emph{bounded E-hulls} (E=exhaustion)
if for any compact set $Y\subset\mathbb C^n$ the set $h(K\cup Y)$ is bounded. 
\end{definition}

\section{The fundamental result of Anders\`{e}n and Lempert}

For a non zero vector $v \in \C^n$ let $\pi_v$ denote the projection of $\C^n$ along $v$ to the
orthogonal complement of $\C v$.

\begin{lemma}\label{andersen}
Let $X$ be a polynomial vector field on $\mathbb C^n$ and let 
$\mathrm{v}\in\mathbb C^n$ be a non zero vector.
Then for any $\epsilon>0$
$X$ can be written as sum $X=\sum_{j=1}^{N_1}X_j+\sum_{j=1}^{N_2}Y_j$ such that 
the following holds.   
\begin{itemize}
\item[(a)] $X_j(z)=f_j(\pi_{v_j}(z))v_j$ with $f_j\in\mathcal(\mathbb C^{n-1})$, $\|v_j-\mathrm{v}\|<\epsilon$, and  
\item[(b)] $Y_j(z)=g_j(\pi_{w_j}(z)))\cdot \langle z,w_j\rangle w_j$ with $g_j\in\mathcal(\mathbb C^{n-1})$,$\|w_j-\mathrm{v}\|<\epsilon.$
\end{itemize} Here we identify $\C^n$ with it's tangent space at each point as usual.
If $\mathrm{div}X=0$ the terms on the form (b) are not needed.
\end{lemma}
We will say that a decomposition of $X$ like this is a decomposition \emph{respecting $(v,\epsilon)$}.

\begin{proof}:  Writing $X$ as a sum of its homogeneous parts it is sufficient to prove the lemma
for $X= X_k$ homogeneous of degree $k$. For a homogeneous polynomial vector field on $\C^n$
can  we apply the following lemma proved in
\cite{KL} (Lemma 7.6).

\begin{lemma}
There exist $n \cdot {n+k-2 \choose n-1}- {n+k-2 \choose n-1}$ linear forms $\lambda_i \in (\C^n)^*$ and vectors $v_i\in \C^n$ with $\lambda_i (v_i)=0$ and
$\Vert v_i\Vert =1 $ $i=1, 2, \ldots,  {n+k-2 \choose n-1}$ together with ${n+k-2 \choose n-1}$
linear forms $\tilde\lambda_j \in (\C^n)^*$ and  vectors $w_j \in \C^n$ with
$\lambda_j (w_j)=0$ and
$\Vert w_j\Vert =1 $ $j=1, 2, \ldots,  n \cdot {n+k-2 \choose n-1}- {n+k-2 \choose n-1}$ such that the homogenous
polynomial maps

$$z\mapsto (\lambda_i (z))^k v_i, \quad  i=1, 2, \ldots, n \cdot {n+k-1 \choose n-1}- {n+k-2 \choose n-1}$$

of degree $k$ together with the homogenous polynomial maps

$$ z\mapsto (\tilde\lambda_j (z))^{k-1} \langle z, w_j\rangle w_j, \quad j=1, 2, \ldots, {n+k-2 \choose n-1}$$

of degree $k$ form a basis of the vector space $V_k \cong  S^k ((\C^n)^*) \otimes \C^n$ of homogenous polynomial maps
of degree $k$. Moreover if $v_0 \in \C^n$ and a non-zero functional $\lambda_0 \in (\C^n)^*$ with $\lambda_0 (v_0)=0$ and
$\Vert v_0\Vert =1 $ and a number $\epsilon >0$ are given, then the vectors $v_i, w_j$ together with the functionals $\lambda_i,\tilde \lambda_j$ can be chosen with
$\Vert v_0 -v_i\Vert  <\epsilon$ $\Vert v_0 -w_j\Vert  <\epsilon$ and $\Vert \lambda_0 - \lambda_i \Vert <\epsilon$, $\Vert \lambda_0 -\tilde \lambda_j \Vert <\epsilon$. 
\label{basis}
\end{lemma}

Remark that the normalization for the vectors $v_i$ and $v_0$ is not important since constants can be moved over to the linear functionals. Now finally remark that if $\lambda$ is a linear
functional with a non zero vector $v$ in it's kernel, then it factors over the projection $\pi_v$.Thus
$(\lambda_i (z))^k = f_i (\pi_{v_i})$.
\end{proof}

\section{Perturbations of families of totally real manifolds}

Without the compact set $K$ the following perturbation result was proved by Forstneri\v{c} and 
Rosay in the case that $M$ is a real analytic surface, and by Forstneri\v{c} 
assuming $dim(M)\leq (2n/3)$.   It was proved without parameters by 
the second author and E. L\o w in \cite{LW}.   The proof we give here 
is very similar to the non-parametric case, building on the original 
idea of Forstneri\v{c} and Rosay.  

\begin{proposition}\label{fixingisotopy}
Let $K\subset\mathbb C^n$ be a compact set, and 
let $M\subset\mathbb C^n$ be a smooth manifold with $dim_{\mathbb R}(M)<n$. 
Let $f:[0,1]\times (K\cup M)\rightarrow\mathbb C^n$ be an isotopy of continuous 
injective maps, let $K\subset U'\subset\subset U$ be 
open sets, and let $G_t:U\rightarrow U_t\subset\mathbb C^n$ be 
a continuous family of homeomorphisms such that $K_t=f_t(K)\subset U_t=G_t(U)$
for all $t$.  Assume that the following hold for all $t$:

\begin{itemize}
\item[1)] $\widehat K_t\subset U'_t:=G_t(U')$, and 
\item[2)] $f_t:\overline{M\setminus K}\rightarrow\mathbb C^n$ is a 
$\mathcal C^k$-smooth isotopy which is an embedding onto 
a totally real manifold. 
\end{itemize}
Then for any $\epsilon>0$ there exist $g_t:[0,1]\times  (K\cup M)\rightarrow\mathbb C^n$ 
and an open set $U''\subset U$
such that 
\begin{itemize}
\item[1)] $g_t|_{(K\cup M)\cap U''}=f_t$ for all $t$, 
\item[2)] $|g_t - f_t|_{x,k}<\epsilon$ for all $x\in M\setminus K$, and 
\item[3)] $\mathrm{h}(g_t(K\cup M))\subset U'_t$ for all $t$.
\end{itemize}
Assuming also that $f_0(K\cup M)$ and $f_1(K\cup M)$ are polynomially convex, 
we may achieve that $g_0=f_0$ and $g_1=f_1$.
\end{proposition}

\begin{proposition}\label{locallycvx}
Let $\Psi:\overline B\rightarrow\mathbb C^k$ be a $\mathcal C^1$-smooth 
map of the form $\Psi(x)=(x,\psi(x))$, where $B$ is the unit ball in $\mathbb R^k$, 
and we write $\mathbb C^k=\mathbb R^k\oplus i\mathbb R^k$.
Assume that $\psi$ is Lipschitz-$\alpha$ with $\alpha<1$.  Then there 
exist $\epsilon,\delta$ such that the following hold:  for any 
$\tilde\Psi$ with $\|\tilde\Psi-\Psi\|_{1,x}<\epsilon$ for all $x\in\overline B$ and any 
point $z_0\in\mathbb C^k\setminus\tilde\Psi(\overline{B})$ with $\pi_x(z_0)\in\pi_x(\tilde\Psi(\overline B_{1/2}))$, 
there exists an entire function $g$ such that 
\begin{itemize}
\item[1)] $g(z_0)=1$, 
\item[2)] $\|g\|_{\tilde\Psi(\overline B)}<1$, and 
\item[3)] $|g(z)|<1$ for all $z$ such that $\mathrm{dist}(z,\tilde\Psi(bB))\leq\delta$.
\end{itemize} 
\end{proposition}
\begin{proof}
We will give the argument considering only the map $\Psi$, and 
it will be clear that it is stable under small perturbations.  
Write $x_0=\pi_x(z_0), z'=x_0+i\psi(x_0)$, and 
consider 
the function $h(z)=(z-z')^2$ on $\Psi(\overline B)$: we 
have $Re(h(z))=Re(((x-x_0)+ i(\psi(x)-\psi(x_0)))^2) = |x-x_0|^2 - |\psi(x)-\psi(x_0)|^2\geq (1-\alpha^2)|x-x_0|^2$. 
Clearly $Re(h(z_0))<0$ and so defining $g(z):=ce^{-h(z)}$ takes care of 1) and 2) for 
a suitable constant $c>0$.  And if $\delta$ is chosen small enough and $\mathrm{dist}(z,\Psi(b(B)))\leq \delta$, 
we have $|\pi_x(z)-x_0|>0$ and $|\pi_y(z)|<|\pi_x(z)|$, and so by the same calculation as above 
we have $|g(z)|<1$.  Clearly these estimates can be made to hold under small perturbations.  
\end{proof}

\begin{corollary}\label{locallycvx2}
Let $(\Psi,\Phi):\overline B\rightarrow \mathbb C^n=\mathbb C^k\times\mathbb C^{n-k}$ be 
a $\mathcal C^1$-smooth map, where $B$ is the unit ball 
in $\mathbb R^k$, and $\Psi$ is as in the previous proposition.  Let $\epsilon$ be as above. 
Then if $|\tilde\Psi - \Psi|_{1,x}<\epsilon$ for all $x\in\overline B$, we have that 
$K:=(\tilde\Psi,\tilde\Phi)(\overline B_{1/2})$ is polynomially convex, where $\tilde\Phi$
is any continuous map. 
\end{corollary}
\begin{proof}
Let $\pi_k$ denote the projection onto $\mathbb C^k$.  Now $\pi_k$ is 
an entire map which maps $K$ onto a polynomially convex totally 
real manifold $S\subset\mathbb C^k$.  Since each point of $S$ is then 
a peak point for the algebra $P(S)$ it follows that $K$ is polynomially convex (see \cite{MWO}).
\end{proof}

\begin{corollary}\label{hullnotinsmallnbh}
Let $K\subset\mathbb C^n$ be a compact set, and let $M\subset\mathbb C^n$
be a compact totally real set. Then for any open neighborhood $U$ of $K$
there exist $\beta,\epsilon>0$ such that the following hold: 
\begin{itemize}
\item[i)] If $S\subset M$ is any closed set, and if $h(K\cup S)\subset (K\cup S)(\beta)$,
then $h(K\cup S)\subset U$, and
\item[ii)] if $M_\epsilon$ is a $\mathcal C^1$-$\epsilon$-pertubation of $M$, and if 
$S\subset M_\epsilon$ is closed, then i) still holds.  
\end{itemize}

\end{corollary}

\begin{proof}
To show i) it is suffices by compactness to show that 
for any point $x\in M\setminus K$ there exist an $r>0$
such that if $\alpha$ is small enough,  then if 
$S\subset M$ is any closed set with $x\in S$, and 
if $h(K\cup S)\subset (K\cup S)(\alpha)$, then 
$h(K\cup S)\cap B_r(x)=\emptyset$.  The argument 
we give will make it clear that this is stable under small 
perturbations of $M$.  \

Fix $x\in M\setminus K$.  By scaling there exist $0<r_1<r_2<<1$ and $\delta>0$ such 
that for any $z_0\in B_{r_1}(x)\setminus M$ there exists 
an entire function $f$ with the following properties
\begin{itemize}
\item[a)] $f(z_0)=1$
\item[b)] $\|f(z)\|<1$ for all $z\in M\cap B_{r_2}(x)$, and 
\item[c)] $\|f(z)\|<1$ for all $z$ with $dist(z,bB_{r_2}(x)\cap M)<\delta$.
\end{itemize}
Now choose $\beta$ so small that if we define 
\begin{itemize}
\item[d)] $U_1:=(K\cup M)(\beta)\cap B_{r_2}(x)$, and 
\item[e)] $U_2:=[(K\cup M)(\beta)\setminus B_{r_2}(x)]\cup [(K\cup M)(\beta)\cap (bB_{r_2}\cap M)(\delta)]$
\end{itemize}
then $\{U_1,U_2\}$ is an open cover of $(K\cup M)(\beta)$ with $U_1\cap U_2\subset (bB_{r_2}\cap M)(\delta)$.  Note that a function $f$ as above will satisfy

\begin{itemize}
\item[f)] $\|f(z)\|<1$ for all $z\in U_1\cap U_2$.
\end{itemize}

Now let $S\subset M$ be any closed set with $x\in S$.   Then $(K\cup S)(\beta)\subset (K\cup M)(\beta)$.  If $h(K\cup S)\subset (K\cup S)(\beta)$ there exists a Runge and Stein 
neighborhood $\Omega\subset (K\cup S)(\beta)$ of $K\cup S$.  We define 
$\tilde U_1=U_1\cap\Omega$ and $\tilde U_2=U_2\cap\Omega$.   Then $\{\tilde U_1,\tilde U_2\}$
is an open cover of $\Omega$.  For any point $x\in B_{r_1}(x)\setminus M$ we 
let $f$ be a function as above.  Regarding $f^m$ as a cocycle  on $\tilde U_1\cap\tilde U_2$
we solve Cousin problems with sup-norm estimates and get by f) holomorphic
functions on $\Omega$ that separate $x$ from $K\cup S$.  \

We now have that $h(K\cup S)\cap B_{r_1}(x)\subset M$ which implies
our claim, since totally real points on polynomially convex compact 
sets are peak points. \

The result is stable of small perturbations of $M$ since the sizes of $r_1,r_2$ and $\delta$
are stable.  

\end{proof}

Fix an open neighborhood $W\subset\mathbb C$ of $I=[0,1]$.  We will now consider 
subvarieties $\Sigma$ and $Z$ of $W\times\mathbb C^{n}$, and 
by $\Sigma_t$ and $Z_t$ respectively, we will mean the fibers 
over a points $t\in W$.

\begin{lemma}\label{push}
Let $M\subset\mathbb C^n$ be a compact $\mathcal C^k$-smooth manifold (possibly with boundary)
of real dimension $m<n$, 
and let $f:[0,1]\times M\rightarrow\mathbb C^n$ be a $\mathcal C^k$-smooth 
isotopy of embeddings such that $f_0=id$ and assume that $f_t(M)$
is totally real for each $t\in [0,1]$.  Then there exists a $\delta>0$ such that 
the following hold: for any point $x\in M$ and (relatively) 
open sets $U\subset\subset V\subset B_\delta(x)\cap M$, 
any variety $\Sigma\subset W\times\mathbb C^n$, 
and any $\epsilon>0$, there exists a hypersurface $Z \subset  W\times\mathbb C^n$, and a $\mathcal C^k$-smooth
isotopy $g:[0,1]\times M\rightarrow\mathbb C^n$
such that 
\begin{itemize}
\item[i)] $|g_t(x)-f_t(x)|_{k,x}<\epsilon$ for all $x\in M, t\in [0,1]$, 
\item[ii)] $g_t(y)=f_t(y)$ for all $y\in M\setminus V, t\in [0,1]$, 
\item[iii)] $g_t(U\cap M)\subset Z_t$ for all $t\in [0,1]$, and
\item[iv)] $dim(Z\cap\Sigma)<dim(\Sigma)$.
\end{itemize}
\end{lemma}

\begin{proof}
We may assume that $f_t$ is defined on $\tilde M\subset\mathbb C^n$
with $M\subset\subset\tilde M$ with $f_t(\tilde M)$ totally real 
for each $t$.    If $\delta$ is small enough we it follows from Corollary \ref{locallycvx2} that 
the following hold: for each $x\in M$ there exists a parametrization 
$\phi:\overline B\rightarrow\tilde M$ of 
$\tilde M$ near $x$, $B$ is the unit ball in $\mathbb R^k$, such that 
$f_t(\phi(\overline B))$ is polynomially convex for all $t$, and 
such that $B_\delta(x)\cap\tilde M\subset\phi(B)$.  
Fix $x$ and choose a cutoff function 
$\chi$ which is identically one on $\overline U$ and compactly supported in $V$.  
We now consider $\mathbb R^k$ to be contained in $\mathbb C^{n-1}\subset\mathbb C^{n-1}\times\mathbb C$.
By  \cite{Forstneric1994} we see 
that $f_t\circ\phi$ is uniformly approximable in $\mathcal C^k$-norm on $\overline B\times I$ 
by a family $F_t$ of 
holomorphic automorphisms of $\mathbb C^n$ also holomorphic in $t\in W$ (see also Lemma \ref{approxwithidentity}). 
Set $Z:=F_t((\mathbb C^{n-1}\times\{0\})\times W)$.    By genericity we may assume that 
(iv) holds.    
Write $\tilde g_t:=F_t\circ\phi^{-1}$, and finally set 
$g_t(x):=f_t(x)+\chi(x)(\tilde g_t(x)-f_t(x))$.
\end{proof}

\emph{Proof of Proposition \ref{fixingisotopy}:} We give the proof first in the case 
of $M$ being an embedded cube $f_0:[0,1]^k\rightarrow\mathbb C^n$. 
The general case follows by covering $M$ by a locally finite 
family of cubes, and succesively using the result for cubes and gluing
(see \cite{LW} for details).
We will prove the result by induction on $k$, 
and we note that the result is obvious for $k=0$.  Assume that 
the result holds for some $k\geq 0$.   To avoid working with too 
many indices we give the argument for passing from $k=1$ to $k=2$, the
passing from $k=0$ to $k=1$ is simpler, and other cases completely similar.

Choose an open set $K\subset V\subset U'$ such that 
$\widehat{{[(K_t\cup M_t)\cap\overline V_t]}}\subset U'_t$ and 
choose an open set $K\subset U''\subset\subset V$.
For $m\in\mathbb N$ we let $\Gamma_m$ denote 
the grid $\Gamma_m=\{x\in [0,1]^2:x_1=j/m \mbox{ or } x_2=j/m, 0\leq j\leq m\}$, 
We let $Q_{ij}$ denote the cube $Q_{ij}=[i/m,(i+1)/m]\times[j/m+(j+1)/m]$.
For small $\beta>0$ we define $Q_{ij}^\beta=[i/m+\beta,(i+1)/m-\beta]\times[j/m+\beta,(j+1)/m-\beta]$
and $\Gamma_m(\beta)=[0,1]^2\setminus\cup_{ij} Q_{ij}^\beta$.  
If $m$ is large enough then if $f_t(Q_{ij})$ is not contained in $V_t$
then $f_t(Q_{ij})$ does not intersect $U_t''$.  
If $m$ large enough then if $S$ is any collection of $n+1$ cubes $Q_{ij}$
then $\mathrm{h}(((K_t\cup M_t)\cap\overline V_t)\cup S_t)\subset U_t'$. 
Note that this still holds if we replace $f_t$ by a small $\mathcal C^1$-perturbation. 
Now by the induction hypothesis we may (by possibly having to perturb $f_t$ slightly)
assume that for any collection of $n+1$ cubes we also have that 
$\mathrm{h}(((K_t\cup M_t)\cap\overline V_t\cup\Gamma_{m,t})\cup S_t)\subset U_t'$
For this we successively use the induction hypothesis and creating 
the grid by attaching a 1-cube to collections of cubes $Q_{ij}$, perturbing the isotopy each time.  
Finally by choosing a small enough $\beta$ we may assume that 
$\mathrm{h}(((K_t\cup M_t)\cap\overline V_t\cup\Gamma_{m,t}(\beta))\cup S_t)\subset U_t'$ (use 
Corollary \ref{hullnotinsmallnbh}).  Finally by Lemma \ref{push}
we may assume that there are parametrized  subvarieties $Z_{i,j,t}=Z(h_{i,j,t})$
of $\mathbb C^n$ with $Q^\beta_{i,j,t}\subset Z_{i,j,t}$ 
for all cubes not completely contained in $V_t$, 
and such that 
for a fixed $t$, any collection of $n+2$ subvarieties with 
distinct indices intersect empty.  

Now fix $t\in [0,1]$, $x\in h(K_t\cup M_t)$, and let $\mu$ be a 
representative Jensen measure for evaluation at $x$.  
Then 
$$
\log|h_{i,j,t}(x)|\leq\int_{K_t\cup M}\log|h_{i,j,t}|d\mu, 
$$
and so if $\mu$ has mass on $Q^\beta_{i,j,t}$ then 
$x\in Z_{i,j,t}$.  So $\mu$ has mass on at most $n+1$ cubes 
together with $\Gamma_{m,t}(\beta)$ and $(K_t\cup M_t)\cap\overline V_t$.  So 
$x\in U_t'$.

\section{Extensions of maps from totally real manifolds}

In this section we show how to obtain the conditions of Theorem \ref{FRCarleman}
starting with embeddings defined only \emph{on} the manifold $M$.   
Our approach is the same as  that of Forstneri\v{c}
and Rosay in \cite{FR} and Forstneri\v{c} \cite{Forstneric1994} but the presence of an additional compact set $K$ complicates things.

\begin{theorem}\label{mainextension}
Let $M\subset\mathbb C^n$ be a compact totally real manifold of 
class $\mathcal C^\infty$ (possibly with boundary) and let $K\subset\mathbb C^n$
be a polynomially convex compact set.  Let $U$ be an open neighborhood of $K$ and 
let $\varphi: [0,1]\times (U\cup M)\rightarrow\mathbb C^n$ 
be a $\mathcal C^\infty$-smooth map such that 
$\varphi_t|_M$ is an embedding 
and $\varphi_t|_U$ is injective holomorphic for 
each fixed $t$.   Assume also that $\varphi_t(K)$ is polynomially convex for each $t.$
Then there exists an (arbitrarily small ) open neighborhood $U'$ of $K$ such 
that for any $\epsilon>0$ and any $k\in\mathbb N$ there exists a $\mathcal C^\infty$-smooth map 
$\Phi: [0,1]\times\mathbb C^n\rightarrow\mathbb C^n$  such that the following hold:
\begin{itemize}
\item[(a)] $\|\Phi_t-\varphi_t\|_{U'}<\epsilon$, 
\item[(b)] $|\Phi_t-\varphi_t |_{x,k}<\epsilon$, $ x \in M\cap U$
\item[(c)] $\Phi_t|_{U'}$ is holomorphic,
\item[(d)] $\Phi_t\in\mathcal H_k(\mathbb C^n,M)$, 
\item[(e)] $\Phi_t(x)=\varphi_t(x)$ for all $x\in M\setminus U$, and 
\item[(f)] $\Phi_t$ is of maximal rank along $M$ 
\end{itemize}
for all $t\in [0,1]$.
\end{theorem}

The content of the following lemma is essentially to be found in \cite{FR} - the 
difference is the claim that the identity extends to the identity. 

\begin{lemma}\label{extension}
Let $M\subset\mathbb C^n$ be a totally real manifold of class $\mathcal C^\infty$, and 
let $\varphi:[0,1]\times M\rightarrow\mathbb C^n$ be a $\mathcal C^\infty$-smooth isotopy
such that $M_t=\varphi_t(M)$ is totally real for each $t\in [0,1]$.  Let $M'\subset M$ be compact
and assume that $\varphi\equiv\mathrm{id}$ near $M'$.  Then for any $k\in\mathbb N$ there exists a $\mathcal C^\infty$-smooth 
map $\tilde\varphi:[0,1]\times\mathbb C^n\rightarrow\mathbb C^n$ such that 
$\tilde\varphi|_M=\varphi, \tilde\varphi\equiv\mathrm{id}$ near $M'$, and $\tilde\varphi_t\in\mathcal H_k(\mathbb C^n,M)$.     
\end{lemma}

\begin{proof}
%NB: As the proof stands now we loose a derivative and need to assume that $k\geq 2$.

The maps $\varphi_t$ already have maximal rank along $M$ and the Cauchy-Riemann equations 
determine (at the level of jets) the extension in the complex tangent directions.  We need 
to extend the maps in the complex normal directions, and the extensions should be $\overline\partial$-flat. 

For each $t\in [0,1]$ let $N_t$ denote the complex normal bundle of the embedded manifold 
$M_t$.  Let $N$ denote the total bundle $N=\underset{t\in [0,1]}{\cup}N_t$, and 
let $\tilde N$ denote the bundle $N_0\times [0,1]$ over $M\times [0,1]$.

 We let $N_e \subset \C^{n+1}$ and 
$\tilde N_e\subset \C^{n+1}$ denote embedded neighborhoods of the zero sections.  These are both generic 
CR-manifolds.  Let $U$ be an open neighborhood of $M'$
such that $\varphi\equiv\mathrm{id}$ on $U$.  There is a natural bundle injection 
$f:\tilde N_0\cup\pi^{-1}(U\times [0,1])\rightarrow N$ since the two bundles are identically 
defined over $M_0\cup (U\times [0,1])$.  By Lemma \ref{bundelmap} the map 
$f$ extends to a bundle isomorphism 
$\tilde f:\tilde N\rightarrow N$, and $\tilde f$ induces CR-isomorphism $\tilde f_e:\tilde N_e\rightarrow N_e$.   
Note that $\tilde f_e$ is the identity near $M'$.   The map $\tilde f_e$ determines a jet along 
$M\times [0,1]$ which has maximal rank and is $\overline\partial$-flat to order $k-1$, and by 
Whitney's extension theorem the jet extends to a map $\tilde\varphi_t$.\end{proof}

Let $M$ be a $\mathcal C^k$-smooth manifold and let $\pi:N\rightarrow M\times I$
be a complex vector bundle of class $\mathcal C^k$ with fiber $\mathbb C^m$.  
Let $N_0$ denote
the bundle $\pi^{-1}(M\times\{0\})$ over $M$, and form the bundle $\tilde N=N_0\times I$.
Denote the projection by $\tilde\pi$.

\begin{lemma}\label{bundelmap}
Let $M' \subset M$ be a compact subset, let $U\subset M$ be an open neighborhood 
of $M'$, and let $f:\tilde N_0\cup\tilde\pi^{-1}(U\times I)\rightarrow N$
be a $\mathcal C^k$-smooth bundle map  ($\tilde\pi=\pi\circ f$) giving an isomorphism from 
$\tilde N_0\cup\tilde\pi^{-1}(U\times I)$ onto the restriction of $N$ to $M \times {0} \cup U \times I$. 
Then there exists a $\mathcal C^k$-smooth  bundle isomorphism $\tilde f:\tilde N\rightarrow N$
($\tilde\pi=\pi\circ\tilde f$) extending 
$f$ on $\tilde N_0\cup\tilde\pi^{-1}(M'\times I)$.
\end{lemma}
\begin{proof}
If $\{U_j\}$ is an open cover of $M\times I$ over which 
the bundles are trivial then $\tilde N$ is represented by a family
$g_{ij}:U_j\rightarrow GL_m(\mathbb C)$ of $\mathcal C^k$-smooth maps 
(transitions from $\pi^{-1}(U_j)$ to $\pi^{-1}(U_i)$), and $N$ is likewise 
represented by a family $h_{ij}:U_{ij}\rightarrow GL_m(\mathbb C)$.   Finding 
an isomorphism $\tilde f:\tilde N\rightarrow N$ amounts to finding local 
maps $\tilde f_j:U_j\rightarrow GL_m(\mathbb C)$ such that 
\begin{itemize}
\item[(a)] $\tilde f_i=h_{ij}\circ\tilde f_j\circ g_{ji}$ for all $U_i\cap U_j\neq\emptyset$.
\end{itemize}
We may think of such an $\tilde f$ as a section of the $\mathcal C^k$-smooth 
$GL_m(\mathbb C)$ fiber bundle $\pi_x:X\rightarrow M\times I$ where the matrices transform 
according to the rule $A\mapsto h_{ij}\circ A\circ g_{ji}$.   Our given bundle injection $f$
is then interpreted as a section of $\pi_x^{-1}(M_0\cup (U\times I))$.  Choose a closed set 
$Y\subset U$ such that  $M'\subset Y^{\circ}$ and such that $(M,Y)$ is a relative CW-complex.  According to
Theorem 7.1 in  \cite{Hu} the section $f$ extends to a section $\tilde f$ of the total bundle $X$. 
By smoothing we may assume that $\tilde f$ is actually a $\mathcal C^k$-smooth section.  
\end{proof}

\emph{Proof of Theorem \ref{mainextension}:}

Choose open subsets $U_j$ in $\mathbb C^n$ for $j=1,2,3$ such that 
$K\subset U_3\subset\subset U_2\subset\subset U_1\subset\subset U$.
The set $U_3$ will play the role of $U'$ in the theorem.  

Note that if $\varphi\equiv\mathrm{id}$ on $U_2$ then 
the theorem follows immediately from Lemma \ref{extension}
by defining $M'=M\cap\overline{U_3}$.  To prove the theorem 
we will use a global holomorphic change of coordinates  
so that we are approximately in this situation.  

By possibly having to choose a smaller $U$ we may assume 
that $\varphi$ is the uniform limit of one parameter families 
$\psi_t\in\mathrm{Aut_{hol}}(\mathbb C^n)$, \emph{i.e.}, we 
may assume that $\psi^\delta_t\rightarrow\varphi_t$ uniformly 
on $[0,1]\times\overline U$ as $\delta\rightarrow 0$.    
Note that the Cauchy-estimates imply: 
\begin{itemize}
\item[(a)] If $\vartheta^\delta_t$ is close enough to the identity in $\mathcal C^k$-norm 
on $\overline{U_3}\cup (M\cap\overline{U_1})$ and if $\psi_t^\delta$  is 
close enough to $\varphi_t$ on $\overline U$, then $\psi^\delta_t\circ\vartheta_t^\delta$
is close to $\varphi_t$ in $\mathcal C^k$-norm on $\overline{U_3}\cup (M\cap\overline{U_1})$.
\end{itemize}

Define $\theta_t^\delta:=(\psi_t^\delta)^{-1}\circ\varphi_t$.  Then $\theta_t^\delta$
converges to the identity uniformly in $\mathcal C^k$-norm as $\delta\rightarrow 0$
on $\overline{U_1}$.   Let $\chi\in\mathcal C^k_0(U_1)$ such 
that $\chi_1\equiv 1$ near $\overline{U_2}$.
 
Write $\theta^\delta_t=\mathrm{id}+\sigma^\delta_t$ and define 
$\tilde\theta^\delta_t:=\mathrm{id}+(1-\chi_1)\cdot\sigma^\delta_t$.  
Then $\tilde\theta^\delta_t\rightarrow\mathrm{id}$ uniformly in $\mathcal C^k$-norm 
on $M\cap\overline U_1$ as $\delta\rightarrow 0$,  $\tilde\theta^\delta_t$ is the identity on $M\cap\overline{U_2}$
and $\tilde\theta^\delta_t=\theta^\delta_t$ outside $U_1$.  Let $M':=M\cap\overline{U_3}$
and let $\vartheta_t^\delta$ be the extensions of $\tilde\theta^\delta_t$
according to Lemma \ref{extension} which now can be extended to 
the identity near $K$.    We set $\Phi_t=\psi_t^\delta\circ\vartheta^\delta_t$
for small enough $\delta$.

$\hfill\square$

\section{A Carleman version of a result by Forstneri\v{c} and Rosay}

\subsection{The nice projection property}

Let $v\in\mathbb C^n$
be a nonzero vector and let $\epsilon>0$.  By $v_\epsilon$
we will mean an arbitrary vector satisfying $\|v_\epsilon-v\|\leq\epsilon$.
We let $\pi_{v_\epsilon}$ denote the orthogonal projection to the orthogonal complement of the vector $v_\epsilon$.     To simplify notation we always write $\mathbb C^{n-1}$ for 
these orthogonal complements, and by $R\mathbb B^{n-1}$ we 
mean $R\mathbb B^n$ intersected with the orthogonal complements.  

Let $M$ be a smooth submanifold of $\mathbb C^n$. We will assume that $M$ satisfies the 
following properties:

\begin{itemize}
\item[$A_1:$] The familly $\pi_{v_\epsilon}:M\rightarrow\mathbb C^{n-1}$ is uniformly proper, \emph{i.e.}, for any compact set $K\subset\mathbb C^{n-1}$, the set $\underset{\mathrm{v}_\epsilon}{\bigcup}\pi_{v_\epsilon}^{-1}(K)$ is compact,  
\item[$A_2:$] there exists a compact set $C\subset M$ such that $\pi_{v_\epsilon}:M\setminus C\rightarrow\mathbb C^{n-1}$ is an embedding onto a totally real manifold,
\item[$A_3:$] the familly $\pi_{v_\epsilon}(M)$ has uniformly bounded E-hulls in $\mathbb C^{n-1}$, \emph{i.e.}, for any compact subset $K\subset\mathbb C^{n-1}$ there exists an $R>0$ such that $h(K\cup\pi_{v_\epsilon}(M))\subset R\cdot\mathbb B^{n-1}$, and 
\item[$A_4:$] for any compact set $K\subset M$ we have that $x\mapsto\langle x,v_{\epsilon_0}\rangle$ is uniformly bounded 
away from zero on $K$ provided $\epsilon_0$ is small enough (depending on $K$).
\end{itemize}

\begin{remark}
 It follows from $A_3$ that $M$ has bounded E-hulls in $\mathbb C^n$.
\end{remark}

\begin{definition}
Let $M'\subset\mathbb C^n$ be a smooth manifold.  If there exists a holomorphic
automorphism $\alpha\in\mathrm{Aut_{hol}}\mathbb C^n$ and a pair $(v,\epsilon)$ such that 
the manifold $M=\alpha(M')$ satisfies $A_1-A_4$ we say that $M'$ has the \emph{nice projection property}.
\end{definition}

\begin{theorem}\label{FRCarleman}
Let $M\subset\mathbb C^n$ be a totally real manifold of class $\mathcal C^\infty$, fix a $\mathcal C^k$-norm on $M$, and assume that $M$
has the nice projection property.  Let $K\subset\mathbb C^n$ be compact, and assume that $K\cup M$ is polynomially convex. Let $\Omega$ be an open neighborhood of $K$.  The following hold:

Let $\phi:[0,1]\times (\Omega\cup M)\rightarrow\mathbb C^n$ be a $\mathcal C^\infty$-smooth map with the following properties:
\begin{itemize}
 \item [$(a)$] $\phi_0$ is the identity map,
\item [$(b)$] $\phi_t|$ is injective for all $t$, 
\item [$(c)$] $\phi_t|_\Omega$ is holomorphic for all $t$, 
\item [$(d)$] $\phi_t|_M$ is an embedding for all $t$, 
\item [$(e)$] there exists a compact set 
$S\subset M$ such that $\phi_t(z)=z$ for all $z\in M\setminus S$ for all $t$,
\item [$(f)$] $\phi_t(K\cup M)$ is polynomially convex for all $t$.
\end{itemize}
Then 
for any strictly positive continuous function $\delta\in\mathcal C(K\cup M)$ there exists a
map $\psi\in\mathrm{Aut_{hol}}\mathbb C^n$,
such that 
$$
|\psi-\phi_1|_{k,x}<\delta(x)
$$
for all $x\in K\cup M$.
\end{theorem}

Preparing for the proof we start with a lemma.  

\begin{lemma}\label{approxwithidentity}
Let $M\subset\mathbb C^n$ be a compact totally real manifold (possibly with boundary) of class $\mathcal C^\infty$, and let $K\subset\mathbb C^n$ be a compact set such that $K\cup M$ is polynomially convex.
Let $A_1\subset A_2\subset\subset M\setminus K$ be closed subsets
with $A_1\subset int(A_2)$.   Let $\Omega$ be an open neighborhood of $K$, and let $\phi:\Omega\cup M\rightarrow\mathbb C^n$
satisfy (a)--(d) and (f) and also that $\phi_t|_{A_2}=id$ for each $t$.

Then there exist open neighborhoods $U'\subset U\subset\Omega$ of $A_1$, 
such that for any $\epsilon>0$, $\delta>0$ sufficiently small, and $k\in\mathbb N$,
there exists $\psi:[0,1]\times\mathbb C^n\rightarrow\mathbb C^n$ with $\psi(t,\cdot)$ holomorphic for each $t$ and real analytic in $t$, 
such that the following hold for all $t\in [0,1]$:
\begin{itemize}
\item [(i)] $\|\psi_t-\phi_t\|_{\overline {K(\delta)}}<\epsilon$, 
\item [(ii)] $|\psi_t-\phi_t|_{k,x}<\epsilon$ for all $x\in M$, 
\item [(iii)] $\|\psi_t-id\|_{\overline U}<\epsilon$, 
\item [(iv)] $\psi_t$ has rank $n$ on $M$, and
\item[(v)] $h(\psi_t(K\cup M\cup \overline U'))\subset U\cup \psi_t(K(\delta))$.
\end{itemize}
\end{lemma}
\begin{proof}
By Corollary \ref{hullnotinsmallnbh} and the assumption (f) there exists a 
$\delta$ small enough such that $h(\overline{K(\delta)}\cup M)\subset\Omega$, so for the purpose of 
approximating $\phi$ on $\overline{K(\delta)}\cup M$ we may assume that $\overline{K(\delta)}\cup M$
is polynomially convex.  If $U'\subset\subset U\subset\subset U''$ are small enough neighborhoods 
of $A_1$, we may extend $\phi$ to be the identity on $U''$, and by the same 
corollary get that 
\begin{itemize}
\item [(iv)] $h(\overline{K(\delta)}\cup M\cup\overline U)\subset U''$, 
\item[(v)] $h(\psi_t(K\cup M\cup \overline U'))\subset U\cup \psi_t(K(\delta))$,
for any $\psi$ which is a sufficiently small perturbation of $\phi$ on $\overline{K(\delta)}\cup M\cup\overline U$.
\end{itemize}

By Theorem \ref{mainextension} we may also assume that $\phi$ is 
$\overline\partial$-flat to order $k$ along $M$, and has rank $n$ along $M$.

It remains to show 
(i)--(iii) and for this we will transform this into an approximation problem without a parameter $t$.  
Set 
$$
K':=\mbox{closure}( \widehat{ [\overline{K(\delta)}\cup M\cup\overline U]} \setminus M), 
$$
define $\tilde M=M\times [0,1]\subset\mathbb C^n\times\mathbb C$ and $\tilde K=K'\times [0,1]\subset\mathbb C^n\times\mathbb C$.  Note that 
$\tilde M$ is a totally real manifold and that $\tilde K\cup\tilde M$ is 
a polynomially convex.
For $N\in\mathbb N$ we define a covering of the interval $I=[0,1]$ as follows: let $I_0=[0,\frac{1}{N})$, let $I_N=(1-\frac{1}{N},1]$, and let $I_j=(\frac{2j-1}{2N},\frac{2j+1}{2N})$ for $j=1,2,...,N-1$.  Let $\{\alpha_j\}$ be a partition of unity with respect 
to the cover $\{I_j\}$.  Define 
$$
\tilde \phi_w(z):=\underset{0\leq j\leq N}{\sum}\alpha_j(w)\cdot \phi_{\frac{j}{N}}(z),
$$
on $\tilde K\cup\tilde M$, with coordinates $(z,w)$ on $\mathbb C^n\times\mathbb C$.
Each of the functions $\alpha_j$ may be approximated arbitrarily well on $I$ by entire functions $\tilde\alpha_j$ on $\mathbb C$.
So the mapping 
$$
\widehat \phi_w(z):=\underset{0\leq j\leq N}{\sum}\tilde\alpha_j(w)\cdot \phi_{\frac{j}{N}}(z),
$$
is in $\mathcal H_k(\mathbb C^{n+1},\tilde M)\cap\mathcal O(\tilde K)$. 
If $N$ was chosen big enough, and if the approximation of the partition of unity was good enough, then
\begin{itemize}
\item [(a)] $\|\widehat \phi_w-\phi_w\|_{K_w}<\frac{\epsilon}{2}$, and
\item [(b)] $|\widehat \phi_w-\phi_w|_{k,x}<\frac{\epsilon}{2}$ for all $x\in M_w$. 
\end{itemize}

\medskip

That we can approximate $\widehat\phi$ on $\tilde K\cup\tilde M$ follows directly from \cite{MWO}.

\end{proof}

\subsection{Proof of Theorem \ref{FRCarleman}}

Choose $r_1>0$ such that $S\subset r_1\cdot\mathbb B^n$.   By possibly having 
to increase $r_1$ we may assume that $\phi_t((M\cap r_1\mathbb B^n)\cup K)\subset r_1\mathbb B^n$
for all $t$.
It follows from $A_3$ above that
there exists an $R>0$ such that 
$$
\overline{h(\pi_{v_{\epsilon}}(\overline{r_1\cdot\mathbb B^{n}}\cup M))}\subset R\cdot\mathbb B^{n-1}
$$
for all $\|v_{\epsilon}-v\|\leq\epsilon$.   
By possibly having to increase $R$ we may assume that  
$\pi_{v_\epsilon}(C)\subset R\cdot\mathbb B^{n-1}$ for all $\|v_{\epsilon}-v\|\leq\epsilon$.
Given $r_2<r_3$ we let $A$ denote 
the annular set 
$$
A=A(r_2,r_3):=\{z\in\mathbb C^n; r_2\leq\|z\|\leq r_3\}.
$$
Fix $r_2<r_3$ such that $\pi_{v_{\epsilon}}(A\cap M)\subset\mathbb C^{n-1}\setminus R\cdot\overline{\mathbb B^{n-1}}$ for all $\|v_{\epsilon}-v\|\leq\epsilon$.
Let $M_j$ denote $M_j:=M\cap\overline{r_j\mathbb B^{n}}$ for $j=2,3$.  Let $\chi\in\mathcal H_k(\mathbb C^n,M)$
such that $0\leq\chi\leq 1$, $\chi|_{r_1\cdot\overline{\mathbb B^n}\cup M_2}\equiv 1$, and $\chi|_{M\setminus M_3}\equiv 0$. 
Let $\tilde A:=M\cap A$. 
Set 
\begin{equation}\label{normchi}
T:=\underset{x\in A}{\sup}\{|\chi|_{k,x}\}.
\end{equation}
By \cite{LW} we have that if $M'$ is a sufficiently small $\mathcal C^1$-perturbation of $\pi_v(\overline{M\setminus C})$
which is equal to $\pi_v(M)$ on $\pi_v(M\setminus M_3)$, then 
$$
h(\pi_v(r_1\cdot\overline{\mathbb B^n})\cup M')\subset R\cdot\mathbb B^{n-1}.
$$
It follows that there exists a constant $\epsilon_1>0$ such that, by possibly having to decrease $\epsilon$, we have that
\begin{itemize}
\item[(1)]\label{pert}
if  $M'$ is a $\mathcal C^1$-$\epsilon_1$-perturbation of $M\setminus C$ which is equal to M outside $M_3$, then 
$\pi_{v_{\epsilon}}(M')$ is a totally real manifold with $h(r_1\cdot\overline{\mathbb B^n}\cup M')\subset R\cdot\mathbb B^{n-1}$.  
\end{itemize}

\subsubsection{Plan of proof}
The theorem will be proved in several steps; the first steps are the same as in the usual A-L procedure.  
\begin{itemize} 
\item[i)] First we will approximate the whole 
isotopy $\phi_t(z)$ on $K\cup M_3$ by an isotopy $g_t(z)$ which is holomorphic on $\mathbb C^n\times [0,1]$.
In addition to being a good approximation on $K\cup M_3$ we need $g_t(z)$ to be uniformly small on a full open neighborhood $U$
of $A$.
\item[ii)] We interpret $g_t(z)$ as the flow of a time dependent vector field $X_t(z)$, which we, by approximation, will assume is polynomial, and then approximate 
$g_1(z)$  by a composition of flows $h^j_{t}(z)$ of time independent polynomial vector fields $X^j(z), j=1,...,m$, all of then being uniformly small 
on $U$.   
\item[iii)] Each $X^j$ may we written as a sum of shear- and over-shear fields.  We will then approximate 
each flow $h^j_{t}$ by a composition of shear- and over-shear flows.  
\item[iv)] Each shear- and over-shear flow from step iii) will be modified on ${r_1\cdot\overline{\mathbb B^n}}\cup M_3$
by multiplying with smooth cutoff functions along (images of) $A$, thereby obtaining good (smooth) shear- and over-shear 
like maps defined on (images of) ${r_1\cdot\overline{\mathbb B^n}}\cup M$, being the identity outside $M_3$.
\item[v)] Finally, the modified maps will be interpreted as shears- and over-shears defined by 
using the projections of ${r_1\cdot\overline{\mathbb B^n}}\cup M$ along the $v_\epsilon$-s (using the good projection property), and approximated by holomorphic shears and over-shears using Carleman appoximation by entire functions.  
\end{itemize}

\subsubsection{Approximation by an isotopy of holomorphic injections}

Let $U'\subset U$ be neighborhoods of $A$ as in Lemma \ref{approxwithidentity}, 
and let $\delta$ be as in the same lemma. 
Let $0<\epsilon_2<\epsilon_1$ be a small constant to be determined later.    
According to Lemma \ref{approxwithidentity} there exists an isotopy $\psi_t$
of entire maps such that 

\begin{itemize}
 \item [(i)] $\|\psi_t-\phi_t\|_{\overline{K(\delta)}}<\epsilon_2$, 
\item [(ii)] $|\psi_t-\phi_t|_{k,x}<\epsilon_2$ for all $x\in M_3$, and 
\item [(iii)] $\|\psi_t-id\|_{\overline U}<\epsilon_2$ for all $t$.
\end{itemize}

Note that 
$$
h(\psi_t(K\cup M_3\cup\overline{U'}))\subset\psi_t(K(\delta)\cup U).
$$

We now proceed to approximate the map $\psi_1$ on 
$K\cup M_3$ and the identity on $M\setminus M_3$.  
We may assume that $\psi_0=id$.  

\subsubsection{The reduction to flows of shear- and over shear fields.}

Let $W$ be a neighborhood of $\overline{K(\delta)}\cup M_3\cup\overline U$ such that 
$\psi_t:\overline W\rightarrow\mathbb C^n$
is a family of injections.   
We define a time dependent vector field $X(t,z)$ by
$$
X(t_0,z_0):=\frac{d}{dt}|_{t=t_0}\psi_t(\psi_{t_0}^{-1}(z_0))
$$ 
for all $z_0\in\psi_{t_0}(W)$.   Then $\psi$ is the flow of $X$.  We let $X_t$ denote the autonomous 
vector field we get by fixing $t$.    
Note that
$$
h(\psi_t(K\cup M_3\cup\overline{U'}))\subset\psi_t(W) 
$$
and so each $X_t$ is the uniform limit of polynomial fields near $\psi_t(K\cup M_3\cup\overline{U'})$. \

Thus Lemma \ref{andersen} and the usual Anders\'{e}n-Lempert construction allows us to find 
automorphisms $\Theta_k(z)$ of $\mathbb C^n, k=1,...,N$, 
such that $\Theta(N)=\Theta_N\circ\Theta_{N-1}\circ\ldots \circ\Theta_1$
approximates $\psi_1$ as well as we want near $K\cup M_3\cup\overline{U'}$ 
Moreover, each $\Theta_k$ is of one of the two forms
\begin{equation}
\Theta_{k}(z)= z + \tau_k(\pi_k(z))\cdot v_k, 
\end{equation}
or 
\begin{equation}
\Theta_{k}(z)= z + (e^{\tau_k(\pi_k(z))}-1)\langle z,v_k\rangle\cdot v_k, 
\end{equation}
where the quantities $\langle z,v_k\rangle$ are bounded away from zero, 
and the functions $\tau_k(\pi_k(z))$ are as small as we like.   Also 
by (iii) above we may assume that each map $\Theta_k$ is as small 
as we like on $\overline{U'}$.

\subsubsection{Approximation  by smooth maps.}\label{smoothapproximation}

We will now describe an inductive procedure how to modify the maps 
$\Theta_{k}$.   We have that  
\begin{equation}
\Theta_{k}(z)= z + \tau_k(\pi_k(z))\cdot v_k, 
\end{equation}
or 
\begin{equation}
\Theta_{k}(z)= z + (e^{\tau_k(\pi_k(z))}-1)\langle z,v_k\rangle\cdot v_k, 
\end{equation}
for entire functions $\tau_k$, depending on wether $\Theta_{k}$ is a shear or an over-shear. 
At any rate, we may write 
\begin{equation}
\Theta_{k}(x)= x + g_k(x)\cdot v_k, 
\end{equation}
defined 
for $x\in\Theta(k-1)(A)$.

Let $\Theta_{0}:=id$, and define inductively 
\begin{align*}
\tilde\Theta_{k}(x) & := x + \chi(\tilde\Theta(k-1)^{-1}(x))\cdot g_k (\Theta(k-1)\circ\tilde\Theta(k-1)^{-1}(x))\cdot v_k\\
& := x + \tilde g_k(x)\cdot v_k, 
\end{align*}
for all $x\in\tilde\Theta(k-1)(A)$, $\tilde\Theta_{k}(x):=\Theta_{k}(x)$ for 
all $x\in K\cup M_3\setminus A$, and $\tilde\Theta_{k}(x):=id$ for 
all $x\in M\setminus M_3$.
Writing 
$\Theta(k)(x) = x + h_k(x)\cdot w_k$
on $A$ it is clear that $\tilde\Theta(k)(x) = x + \chi(x)\cdot h_k(x)\cdot w_k$, and 
so this is well defined if each composition $\Theta(k)$ is small enough on $\overline{U'}$. 
Note that $\Theta(N)=x+h_N(x)\cdot w_N$, and that $|h_N\cdot w_N|$ is as small as we 
like depending on the choice of $\epsilon_2$ above.  The choice of $\epsilon_2$
is made after fixing $\chi$, so we choose it depending on the constant $T$ above, and so we
may assume that $\tilde\Theta(N)$ is 
as close to the identity as we like on $A$.  \

\

Finally we want to rewrite the $\tilde\Theta_{k}$-s as shears and over-shears, \emph{i.e.}, defined 
via the projections $\pi_{k}$.  First assume that the original $\Theta_{k}$ was a shear map: 
if $\epsilon_2$ was chosen small enough we have that 
$\pi_k(\tilde\Theta(k-1)(A))$ is a totally real manifold contained in $\mathbb C^{n-1}\setminus R\cdot{\mathbb B^{n-1}}$, so we may write $\tilde g_k(x)=\tilde\tau_k(\pi_k(x))$ on $\Theta(k-1)(A)$.
Since $\tilde\tau_k$ will agree with $\tau_k$ near $\pi_k(\tilde\Theta(k-1)(A\cap\overline{(M_3\setminus A)}))$, 
we may extend $\tilde\tau_k$ to be equal to the original $\tau_k$ on $R\cdot\overline{\mathbb B^{n-1}}\cup
\pi_k(\Theta(k-1)(M_3\setminus A))$.  We may aslo extend $\tilde\tau_k$ to be zero 
on $\pi_k(\Theta(k-1)(M\setminus M_3))$.  

If $\Theta_{k}$ was an over-shear we write first $\tilde g_k(x)=\psi_k(\pi_k(x))$, we extend $g_k$
as we just did with $\tilde\tau_k$, but now we want to solve 
\begin{equation}
(e^{\tilde\tau_k(\pi_k(x))}-1)\langle x,v_k\rangle = \psi_k(\pi_k(x))
\end{equation}
on $\pi_k(\Theta(k-1)( A))$.  This is doable since $\langle x,v_k\rangle$ is uniformly bounded away from zero (independent of $k$)
by $A_4$ in the definition of the nice projection property, and we may assume that $ \psi_k(\pi_k(x))$ is arbitrarily small compared to this. 

\subsubsection{Approximation by holomorphic automorphisms}

We finally show by induction on $k$ that the compositions $\tilde\Theta(k)$ be 
be approximated in the sense of Carleman on $K\cup M$.  \

More generally than showing this first for $k=1$ we show first that 
for any $k$ we have that $\tilde\Theta_k$ may be approximated 
in the sense of Carleman on $\tilde\Theta(k-1)(K\cup M)$.
Note that $h(\pi_1(\tilde\Theta(k-1)(K\cup M)))\subset R\cdot\mathbb B^{n-1}$ and so 
it follows by \cite{MWO} that the function $\tilde\tau_k$ may be 
approximated in the sense of Carleman on $\pi_k(\tilde\Theta(k-1)(K\cup M))$ by entire functions.  

Now the induction step is clear: since $\tilde\Theta_{k+1}$ may be approximated on $\tilde\Theta(k)(K\cup M)$
and since $\tilde\Theta(k)$ may be approximated on $K\cup M$ we get that $\tilde\Theta(k+1)$ may 
be approximated on $K\cup M$.

\section{Approximation of smooth automorphisms of $\mathbb{R}^k\subset\mathbb C^n$}

\begin{theorem}\label{approximationaut}
 Let $\phi:\mathbb R^s\rightarrow \mathbb R^s$ be a $\mathcal C^\infty$-smooth 
automorphism, and assume that $s<n$.
Then $\phi$ can be approximated in the sense of Carleman by holomorphic automorphisms of $\mathbb C^n$, \emph{i.e.}, 
given $\epsilon\in\mathcal C(\mathbb R^s)$ and $k\in\mathbb N$, there exists  $\Psi\in Aut_{hol}\mathbb C^n$
such that $|\Psi-\phi|_{k,x}<\epsilon(x)$ for all $x\in\mathbb R^s$.  
\end{theorem}

We start by describing a gluing procedure that will be used 
in an induction argument to prove Theorem \ref{approximationaut}.
Let $M\hookrightarrow\mathbb C^n$ be a smooth embedded submanifold, and 
let $\pi:N\rightarrow M$ be an embedded neighborhood of the zero section 
of the normal bundle.  Then any sufficiently small $\mathcal C^k$-perturbation $M'$
of $M$ can be thought of as $\mathcal C^k$-small section $s\in\Gamma (M,N)$.   
Fix a normal exhaustion $K_j\subset K_{j+1}^\circ$ of $M$, and 
fix functions $\chi_j\in\mathcal C^\infty_0(K_{j+1}^\circ)$
with $\chi_j\equiv 1$ near $K_j$.  
\begin{lemma}\label{glue}
Let $\psi$ be a smooth diffeomorphism of $M$, and let $\epsilon\in\mathcal C(M)$
be a strictly positive function.  Then there exists a strictly positive $\delta\in\mathcal C(M)$   
such that the following hold.   For any $m\in\mathbb N$ and any smooth embedding $\phi:M\rightarrow\mathbb C^n$ such that $|\phi-\psi|_{k,x}<\delta(x)$ for all $x\in K_{m+1}$, and such that $\phi(M)$ is a $\delta$-perturbation
of $M$ in the sense that $\phi(M)$ can be written as a section $s\in\Gamma(M,N)$
where $|s-\mathrm{id}|_{k,x}<\delta(x)$ for all $x\in M$, the map 
$$
\tilde\phi:= \phi(x) + (1-\chi_m(x))(s(\psi(x))-\phi(x))
$$
satisfies $|\tilde\phi-\psi|_{k,x}<\epsilon(x)$ for all $x\in M$.
\end{lemma}  
\begin{proof}
We check different parts of $M$.   Since on $K_j$ we have that 
$\tilde\phi=\phi$ it suffices that $\delta(x)<\epsilon(x)$ for all $x\in M$.
On $M\setminus K_{m+1}$ we have that $\tilde\phi(x)=s(\psi(x))$, 
and given any $\tilde\epsilon<\epsilon$, 
it is clear that if $\delta$ decreases rapidly towards zero, then 
$|\psi-s\circ\psi|<\tilde\epsilon(x)$ for all $x\in M$.  
Finally there exist constants $C_m, m\in\mathbb N,$ such that 
we have 
$$
|\tilde\phi-\psi|_{k,x}\leq |\phi - \psi|_{k,x} + C_m|s(\psi(x)-\phi(x))|_{k,x}<\delta(x)+C_m\tilde\epsilon(x),
$$
for all $x\in K_{m+1}\setminus K_m$, and so it is clear that the claim holds 
if $\delta$ decreases rapidly as $x$ tends to infinity.  

\end{proof}

Theorem \ref{approximationaut} will be proved by an inductive argument where the main step 
will be covered by the following lemma.

\begin{lemma}\label{mainstepaut}
Let $\psi\in\Aut_{hol}\mathbb C^n$ such that $M=\psi(\mathbb R^s)$
is a sufficiently small $\mathcal C^1$-perturbation (in 
the sense of Carleman) of $\mathbb R^s\subset\mathbb C^n, s<n$, and let 
$K\subset\mathbb C^n$ be a compact set such that $K\cup M$ is holomorphically convex.  Assume 
given $R>0$ such that $K\subset R\mathbb B^n$ and a $\phi\in \mathrm{Diff}(M)$
such that $\phi$ is orientation preserving and $\phi=id$ near $K\cap M$. 
Then for any $k\in\mathbb N, \mu>0$ and strictly positive $\delta\in\mathcal C(M)$ there exist (arbitrarily large) $l\in\mathbb N$ and $\sigma\in Aut_{hol}\mathbb C^n$ such that the following hold
\begin{itemize}
\item[1)] $|\sigma-\phi|_{k,x}<\delta(x)$ for all $x\in \psi(\mathbb R^s\cap (l+1)\overline{\mathbb B^n})$, 
\item[2)] $\|\sigma-id\|<\mu$ near $K$, 
\item[3)] $R\mathbb B^n\subset \sigma\circ\psi(l\mathbb B^n)$, and
\item[4)] $\sigma(M)$ is a $\delta$-$\mathcal C^k$-small perturbation of $M$. 
\end{itemize}

\end{lemma}

\begin{proof}

Note that by \cite{LW} we have that if $M$ is a sufficiently small $\mathcal C^1$-pertur\-bation 
of $\mathbb R^s$ then $h(R\overline{\mathbb B^n}\cup M)\subset (R+1)\mathbb B^n$
for any $R>0$.

Choose a compact set $C\subset M$ such that $\phi(M\setminus C)\subset M\setminus (R+1)\mathbb B^n$.
Let $X(t,x), t\in [0,1]$ be a non autonomous smooth vector field such that 
$\phi$ is the time one map of $X$, and such that $X(t,x)=0$ on $M\cap K$. Denote this flow by $\phi_t$.      Choose a compact 
set $C'$ such that $C'$ contains the complete $\phi_t$-orbit of $C$.   Choose a smooth 
cutoff function $\chi\in\mathcal C^\infty_0(M)$ such that $\chi\equiv 1$ near $C'$.  
Define $\tilde X(t,x):=\chi(x)\cdot X(t,x)$, and let $\tilde \phi_t$ denote the 
flow of $\tilde X$.      By Theorem \ref{FRCarleman}
there exists $\Phi_1\in Aut_{hol}\mathbb C^n$ such that 
$\Phi_1$ approximates $\tilde \phi_1$ on $M$ in the sense of Carlemann, hence
also $\phi$ on $C$, and 
$\Phi_1$ approximates the identity near $K$.   We set $M_1=\Phi_1(M)$.

Now choose $l>>0$ such that $R\mathbb B^n\subset\Phi_1\circ\psi(l\mathbb B^n)$.
Let $\tilde\sigma\in\mathrm{Diff}(M_1)$ be defined by $\tilde\sigma:=\pi_1\circ\phi\circ\Phi_1^{-1}$.  
Note that $\tilde\sigma$ is close to the identity on $\Phi_1(C)$, so 
after a small perturbation we may assume that $\sigma$ is the identity on $\Phi_1(C)$
and furthermore $\sigma$ can then be extended to the identity on $(R+1)\mathbb B^n$
which contains $h(R\cdot\overline{\mathbb B^n}\cup\Phi_1(M))$.

By an argument similar to that above there exists $\Phi_2\in Aut_{hol}\mathbb C^n$ that 
approximates $\tilde\sigma$ on $\Phi_1\circ\psi(M\cap (r+1)\cdot\overline{\mathbb B^n})$
and it is near the identity on $R\mathbb B^n$. Now the composition 
$\psi:=\Phi_2\circ\Phi_1$ furnishes a desired map.  

\end{proof}

\emph{Proof of Theorem \ref{approximationaut}:}
After possibly having to compose with the map $(z_1,z_2...,z_n)\mapsto (-z_1,z_2,...,z_n)$
we may assume that $\phi$ is orientation preserving, and we may also assume that $\phi(0)=0$.

For $i=0,1,2,...$, we will inductively construct sequences of automorphisms $\psi_i \in  Aut_{hol}\mathbb C^n$, real numbers $R_i\leq r_i$, 
$R_i\rightarrow\infty$ as $i\rightarrow\infty$, and diffeomorphisms $\phi_i\in\mathrm{Diff}(\psi_i(\mathbb R^k))$
such that the following hold for $i\geq 1$:

\begin{itemize}
\item[$(1_i)$] $|\psi_i-\phi|_{k,x}< \frac 1 2 \epsilon(x)$ for all $x\in \R^s  \cap r_{i} \overline{\B^n}$
\item[$(2_i)$] $|\phi_i\circ\psi_i -\phi|_{k,x}< \frac 1 2 \epsilon(x)$ for all $x\in \R^s$,
\item[$(3_i)$] $\|\psi_i-\psi_{i-1}\|_{\psi_{i-1}(r_{i-1}\overline{\mathbb B^n})} < (\frac{1} {2})^i$,   
\item[$(4_i)$] $R_i\mathbb B^n\subset\psi_i(r_i\mathbb B^n)$, and
\item[$(5_i)$] $\phi_i=\mathrm{id}$ near $\psi_i(\mathbb R^s\cap r_i\overline{\mathbb B^n})$.
\end{itemize}

In addition, each $\psi_i(\mathbb R^s)$ will be a sufficiently small perturbation of $\mathbb R^s$
such that Lemma \ref{mainstepaut} applies.  \

If we set $r_0=r_1=R_0=R_1=0$, $\psi_0=\psi_1=\phi_0=\mathrm{id}$ we 
get $(1_1)-(4_1)$, and we perturb $\phi$ slightly near the origin to get 
a $\phi_1$ such that $(5_1)$ also holds.   \

To complete the induction step we now assume that $(1_i)-(5_i)$ hold for some $i\geq 1$.  
Choose $R_{i+1}\geq R_i+1$ such that $\psi_i(r_i\mathbb B^n)\subset R_{i+1}\mathbb B^n$.
For any strictly positive $\delta\in\mathcal C(\psi_i(\mathbb R^s))$ there exists
by Lemma \ref{mainstepaut} a $r_{i+1}>R_{i+1}$ and a  $\sigma\in\Aut_{hol}\mathbb C^n$ approximating $\phi_i$ $\delta$-well on $\psi_i(\mathbb R^s\cap (r_{m+1}+1)\mathbb B^n)$, 
and so that setting $\psi_{i+1}:=\sigma\circ\psi_i$ we get
$(1_{i+1}), (3_{i+1})$ and $(4_{i+1})$.   Using Lemma \ref{glue}
we may also achieve that we get a map $\tilde\phi_i:\psi_{i}(\mathbb R^s)\rightarrow\psi_{i+1}(\mathbb R^s)$  such that setting $\phi_{i+1}:=\tilde\phi_m\circ\sigma^{-1}$
gives us $(2_{i+1})$ and $(5_{i+1})$.

\medskip

It now follows from $(3_i)$ that the sequence $\Psi:=\lim_{j\rightarrow\infty}\psi_i$
converges uniformly on $\mathbb C^n$, and we may assume that 
the limit is injective holomorphic.   Moreover it follows from 
$(4_i)$ that the sequence $\psi_i^{-1}$ also converges on $\mathbb C^n$, 
hence $\Psi\in\Aut_{hol}\mathbb C^n$.  By $(1_i)$ we have 
that $\Psi$ is a good enough approximation on $\mathbb R^s$.

$\hfill\square$

\medskip

\bibliographystyle{amsplain}

\end{document}